\documentclass[leqno]{amsart}
\usepackage[none]{hyphenat}
\usepackage[english]{babel}
\usepackage{amssymb}
\usepackage{mathrsfs}
\usepackage{amsmath, amsfonts, vmargin, enumerate}
\usepackage{graphicx}
\usepackage{color}
\usepackage{verbatim}
\usepackage{amsthm}
\usepackage{latexsym, bm}
\usepackage[utf8]{inputenc}
\usepackage[T1]{fontenc}
\usepackage{hyperref}
\usepackage{physics}
\usepackage{euscript}
\usepackage{dsfont}
\usepackage{xeCJK}

\input xypic

\xyoption {all}

 \makeatletter

\newtheorem{thm}{Theorem}

\newtheorem{cor}[thm]{Corollary}

\newtheorem{defi}[thm]{Definition}
\newtheorem{remark}[thm]{Remark}
\newtheorem{example}[thm]{Example}
\newtheorem{conj}[thm]{Conjecture}

\newenvironment{rk}{\begin{remark}\rm}{\end{remark}}

\numberwithin{equation}{section}

\newcommand{\nat}{{\mathbb N}}

\newcommand{\C}{{\mathcal C}}

\newcommand{\N}{{\mathbb N}}

\newcommand{\wt}{\widetilde}

\newcommand{\n}{\noindent}

\newcommand{\be}{\begin{align*}}
\newcommand{\ee}{\end{align*}}
\newcommand{\beq}{\begin{equation}}
\newcommand{\eeq}{\end{equation}}
\newcommand{\beqn}{\begin{equation*}}
\newcommand{\eeqn}{\end{equation*}}

\begin{document}

\title[Finite de Finetti theorems for free easy quantum groups ]{Finite de Finetti theorems for free easy quantum groups }


\author[Jianquan Wang]{Jianquan Wang}
\address{Institute for Advanced Study in Mathematics, Harbin Institute of Technology,  Harbin 150001, China}
\email{wjquan@stu.hit.edu.cn}

\date{}
\maketitle

 \begin{abstract}
We prove various finite de Finetti theorems for non-commutative distributions which are invariant under the free easy quantum group actions. This complements the free de Finetti theorems by Banica, Curran and Speicher, which mostly focus on infinite sequences. We also discuss some refined results for the infinite setting. 
 \end{abstract}


\section{Introduction}\label{Introduction}


 Symmetries can be characterized by the invariance under the actions of specific families,
 yielding a series of results and tools.
 One of the phenomena studied on a sequence $(X_1, X_2,\ldots)$ of random variables is the de Finetti theorem \cite{bk5}, which illustrates the link between symmetries and independence. It demonstrates that the distribution of $(X_1,\ldots,X_n)$ is identical to that of  $(X_{\pi(1)},\ldots,X_{\pi(n)})$ for all $n\geq 1$ and all permutation $\pi$ on $[n]=\{1,\ldots,n\}$ if and only if $X_i$, $i\in\N$, are conditionally independent and identically distributed with respect to the tail algebra.
 An extension by Freedman \cite{ref10} established a de Finetti theorem for rotatable sequences of random variables (i.e. the joint distribution of the random variables is invariant under the action of orthogonal group). 
 
 In the setting of finite sequences,
 Diaconis \cite{ref16} proved two finite forms of de Finetti theorem for exchangeable random variables with (0,1)-distribution through geometric methods. 
 Diaconis and Freedman \cite{ref17} extended the theorem to a general form, obtaining a finite de Finetti theorem for exchangeable random variables taking values in a finite set. These results indicate that
 the joint distribution of a finite exchangeable sequence approximates quantitatively a  finite conditionally independent identically distributed sequence.
 
 Free probability studies noncommutative distributions arising from random matrices and operator algebras. In this context, it is natural to consider quantum symmetries instead of the aforementioned classical ones, which can be characterized by the invariance under the actions of quantum groups. The easy quantum groups, constructed using Woronowicz's Tannaka-Kre\u\i n theorem via
 combinatorial data, were introduced by Banica and Speicher in \cite{ref2}. Moreover, free easy quantum groups are the corresponding quantum groups associated with the categories of non-crossing partitions. Köstler 
 and Speicher \cite{ref7} initiated the study of  de Finetti theorems in the context of  free probability theory, showing that if the non-commutative distribution of a sequence $(x_i)_{i\in\mathbb N}$ is invariant under the quantum permutation group action, 
 then the sequence exhibits operator-valued free independence.
 This theory was later
 largely extended for all free easy quantum groups by Banica, Curran and Speicher \cite{ref5}; see also \cite{ref19,ref20,ref21,ref22,ref23,ref24} for related results.  
 
 However, very few free de Finetti theorems for quantum group actions on finite
 sequences are known so far. The only isolated result was given in the aforementioned work \cite[Theorem 4.8]{ref5}, where a finite approximation of operator-valued moments
 was provided. 
 Recently, 
Baraquin et al. \cite{ref4} established a finite de Finetti theorem for Voiculescu's unitary dual group, offering a complete characterization of the associated invariant distributions by using various vanishing properties of cumulants. This inspired us to investigate a more comprehensive finite de Finetti theorem for
free easy quantum groups in a similar spirit.

In this paper, we establish a finite de Finetti theorem for free easy quantum groups  using Weingarten calculus and moment-cumulant formula. Let $G$ be a free easy quantum group and let $\mathcal C$ be the category of non-crossing partitions associated with $G$. Then, 
for a $\ast$-probability space $(A, \varphi)$ and a sequence $(x_1, \ldots , x_n)$ in $A$ on which $G$ acts, we prove that
the following are equivalent (see Theorem \ref{th1}):
\begin{enumerate}[\rm(i)]
	\item The sequence $(x_1, \ldots, x_n)$ is $G$-invariant. 
	\item  For each $m>0$, there are scalars $(c_\pi^{(m)})_{\pi\in \mathcal C(m)}\subset\mathbb C$ such that, for all $\bm{i}\in [n]^m$,
	\begin{equation*}
		\varphi(x_{i_1}\ldots x_{i_m})=
		\sum\limits_{\substack{\pi\in\mathcal C(m)\\ \pi\leq \ker(\bm{i})}}c_\pi^{(m)}. 
	\end{equation*}
	\item For each $m>0$, there are scalars $(C_\pi^{(m)})_{\pi\in \mathcal C(m)}\subset\mathbb C$ such that, for all $\bm{i}\in [n]^m$,
	\begin{equation*}
		\kappa_{m}(x_{i_1},\ldots ,x_{i_m})=
		\sum\limits_{\substack{\pi\in\mathcal C(m)\\ \pi\leq \ker(\bm{i})}}C_\pi^{(m)}.
	\end{equation*}
\end{enumerate}
The distinguished scalars $c_\pi^{(m)}$ and $C_\pi^{(m)}$, $\pi\in\mathcal C(m)$, can be reversely determined by the moments and Weingarten functions. In view of the free moment-cumulant formulas, we regard
 Assertion (ii) as a weak version of the free independence, which will be explained in Section \ref{sect3}. The last assertion provides a description similar  to the vanishing property of cumulants in \cite[Theorem 4.2]{ref4}. 

This paper is organized as follows. Section \ref{sect2} provides a review of the notations and definitions used throughout the paper. In Section \ref{sect3}, 
we prove our finite de Finetti theorem for free easy quantum groups. As a corollary, we consider the asymptotic freeness of sequences whose distribution is invariant under the action of $O_n^+$ and $S_n^+$. 
Section \ref{sect4} extends our investigation to an infinite setting, following from Theorem \ref{th1}, as shown in the proof of Theorem \ref{th2}. Combining \cite[Theorem 5.3]{ref5} with Theorem \ref{th2}, we can use free cumulants to characterize operator-valued freeness for an infinite sequence, as shown in Corollary \ref{coro3}. Additionally, our arguments can be extended to a unitary version, as explained in Remark \ref{rk3}.


\section{Preliminaries}\label{sect2}


In this section, we briefly review some relevant theories and notation about free probability theory and free easy quantum groups. 
We refer to \cite{bk1} for more on free probability and to \cite{ref6} for additional information on quantum groups.

For $m\in \mathbb N$, we denote by $[m]$ the set $\{1,2,\dots,m\}$.
We call $\pi$ a \textbf{partition of} ${[m]}$ if $\pi=\{V_i\}_{i\in I}$ is a collection of subsets of $[m]$, such that $\bigcup_{i\in I} V_i=[m]$ and $V_i\cap V_j=\varnothing$ for $i\neq j$. We call $V_i$ $(i\in I)$ the \textbf{blocks} of $\pi$ and we denote by $\# \pi$ the number of blocks in $\pi$.  If there do not exist integers $1\leq i<j<k<l\leq m$ and blocks $U,V\in\pi$ such that $U\neq V$, $i,k\in U$ and $ j,l\in V$, then  $\pi$ is said to be \textbf{non-crossing}.  We denote by $P(m)$ all partitions of $[m]$ and by $\operatorname{NC}(m)$ all non-crossing partitions of $[m]$. For $\pi,\sigma\in P(m)$, we write $\pi\leq \sigma$ if each block of $\pi$ is contained in some block of $\sigma$. Given a sequence of integers  $\bm{i}=(i_1,\ldots,i_m)$,  we denote $\pi|_V=\{U\cap V: U\in \pi \}$ and $\bm{i}|_V=(i_{j_1},\ldots,i_{j_l})$ for $V=( j_1,\ldots,j_l)\subset [m]$. We consider $\bm{i}$ as a function from $[m]$ to $[n]$. Let $R(\bm{i})$ be the range of $\bm{i}$, and let $V_{a}=\bm{i}^{-1}(a)$ for all $a\in [n]$.
We denote by $\ker(\bm{i})$ the partition $\{V_a:a\in R(\bm{i})\}$. 

Let $A$ be a $\ast$-algebra and let $\varphi$ be a state on $A$, namely, $\varphi$ is a positive linear functional on $A$ with $\varphi(\bm{1})=1$. Then we call $(A,\varphi)$ a $\bm{\ast}$\textbf{-probability space}. The elements in $A$ are called \textbf{non-commutative random variables}.  We denote by $\kappa_m:A^m \rightarrow \mathbb C$ the \textbf{free cumulant}, which is a multi-linear functional determined by the {moment-cumulant formula} (see, e.g., \cite[Chapter 2]{bk1}):
\begin{equation}
	\label{Emc}
	\varphi(x_{i_1}\cdots x_{i_{m}})=\sum\limits_{\pi\in \operatorname{NC}(m)}\kappa_\pi(x_{i_1},\dots, x_{i_{m}}),
\end{equation}
where $\kappa_\pi : A^m\rightarrow \mathbb C$ is a multi-linear functional defined as
\begin{equation}\label{Emc2}
	\kappa_\pi(x_{i_1}\cdots x_{i_{m}})=\prod_{\substack{V\in\pi\\ V=(j_1,\ldots,j_l)}}\kappa_{l}(x_{i_{j_1}},\ldots,x_{i_{j_l}}).
\end{equation}
We will also use frequently the following fact for Möbius functions (see e.g. \cite[Proposition 10.11]{bk4}).
Let $\mathcal P$ be a finite poset, and let $\mu_{\mathcal P}$ be the Möbius function on $\mathcal P$. For two functions $f,g:\mathcal P\rightarrow\mathbb C$, the relation 
\begin{equation}\label{eqi1}
	f(\pi)=\sum_{\sigma\in \mathcal P,\sigma\leq\pi}g(\sigma),\ \forall \pi\in \mathcal P
\end{equation} is equivalent to
\begin{equation}\label{eqi2}
	g(\pi)=\sum_{\sigma\in \mathcal P,\sigma\leq\pi}f(\sigma)\mu_{\mathcal P}(\sigma,\pi).
\end{equation}

By the equivalence between \eqref{eqi1} and \eqref{eqi2}, we obtain the cumulant-moment formula
\begin{equation}
	\label{Ecm}
	\kappa_m(x_{i_1},\dots, x_{i_{m}})=\sum\limits_{\pi\in \operatorname{NC}(m)}\varphi_\pi(x_{i_1},\dots, x_{i_{m}})\mu_{P(m)}(\pi,1_m),
\end{equation}
where $1_m:=\{\{1,\dots, m\}\}\in P(m)$, $\mu_{P(m)}$ is the Möbius function on $P(m)$ and $\varphi_\pi : A^m\rightarrow \mathbb C$ is a multi-linear functional defined as
\begin{equation*}
	\varphi_\pi(x_{i_1},\dots ,x_{i_{m}})=\prod_{\substack{V\in\pi\\ V=(j_1,\ldots,j_l)}}\varphi(x_{i_{j_1}}\cdots x_{i_{j_l}}).
\end{equation*}
We recall below some basic notions of free easy quantum groups which can be found in \cite[Definiton 1.1, Definiton 3.5]{ref2} and \cite[Section 1]{ref13}.
For $n>0$, let $A_n$ be a unital $C^*$-algebra generated by $n^2$ self-adjoint elements $u_{ij}$ $(1\leq i,j\leq n)$. The pair  $(A_n,u)$ is called  a \textbf{quantum orthogonal group} if for all $1\leq i,j\leq n$,
$$\sum_{k=1}^{n}u_{ik}u_{jk}=\delta_{ij}=\sum_{k=1}^nu_{ki}u_{kj}.$$
We denote $G_n=(A_n,u)$ and $A_n=C(G_n)$.
Let $\Delta: A_n \rightarrow A_n\otimes A_n$ denote the unital $\ast$-algebra homomorphism  defined by 
$$
\Delta(u_{ij})=\sum_{k=1}^n u_{ik}\otimes u_{kj}.
$$
It can be shown that
$(\Delta\otimes \operatorname{id})\circ\Delta= (\operatorname{id}\otimes\Delta )\circ\Delta$. We call $\Delta$ the \textbf{comultiplication} of $G_n$.

Recall that, by Woronowicz's Tannakain duality theorem \cite{ref12}, a quantum orthogonal group $G_n$ can be reconstructed from the intertwiner spaces $C_{G_n}$, where
$$C_{G_n}(k,l)=\qty{T: (\mathbb C^n)^{\otimes k}\rightarrow (\mathbb C^n)^{\otimes l}: T \text{ is linear and } Tu^{\otimes k}= u^{\otimes l}T },$$
where $u$ is the fundamental representation of $A_n$, namely $u$ is the $n\times n$-matrix with entries $u_{ij}$ ($1\leq i,j\leq n$). 
We denote by $P$ the category of all partitions.  For each $k,l>0$, we draw $k$ points in an upper row and $l$ points in a lower row, and we denote by $P(k,l)$ the set consists of all possible line matchings between these points. Let $\bm{i}\in [n]^k$ and $\bm{j}\in [n]^l$ be two sequences of indices associated with  the upper and lower points, respectively. For $p\in P(k,l)$, we set $\delta_{p}(\bm{i},\bm{j})=1$ if whenever two points are connected by $p$,  the indices with respect to the points are equal. Otherwise, we set $\delta_{p}(\bm{i},\bm{j})=0$.  
Let $e_i$, $1\leq i\leq n$, be the canonical basis of $\mathbb C^n$. We denote by $T_p$ the linear map from $\qty(\mathbb C^n)^{\otimes k}$ to $\qty(\mathbb C^n)^{\otimes l}$ defined by
\begin{equation}\label{EqTP}
	T_p(e_{i_1}\otimes\cdots\otimes e_{i_k})=\sum_{j_1,\ldots,j_l=1}^n\delta_p(\bm{i},\bm{j})e_{j_1}\otimes\cdots\otimes e_{j_l}.
\end{equation}
Let $G_n$ be a
 quantum orthogonal group. We call $G_n$ an \textbf{easy quantum group} if its intertwiner space is of the form 
$$C_{G_n}=\operatorname{span}\qty{T_p: p\in P_{G_n} \text{, where } P_{G_n}\subset P \text{ is a category of partitions}}.$$
And $G_n$ is called \textbf{free} if $P_{G_n}$ is a category consisting of some non-crossing partitions. Such categories and free easy quantum groups have been fully classified in \cite{ref13}. These quantum groups are commonly denoted by $O_n^+$, $H_n^+$, $B_n^+$, $S_n^+$, $S_n^{\prime+}$, $ B_n^{\prime+}$ and $B_n^{\#}$.  
 In the remainder of the paper,  the quantum group $G_n$ will always refer to a free easy quantum group and we denote by $\mathcal C$ the category of non-crossing partitions associated with $G_n$. Furthermore, for $m>0$, we set $\mathcal C(m)=\mathcal C(m,0)$.

Let $(A,\varphi)$ be a $\ast$-probability space generated by $n$ self-adjoint elements $(x_1, \ldots, x_n)$.  
We denote by $\alpha_n: A\rightarrow  A\otimes  C(G_n)$ the unique unital homomorphism  defined as
\begin{equation*}\label{eqac}
	\alpha_n(x_i)=\sum_{j=1}^n x_j\otimes u_{ji},
\end{equation*}
for $1\leq i\leq n$.
Then $\alpha_n$ is a {coaction} of $G_n$. And we say that \textbf{the distribution
	${\varphi}$ is invariant under ${G_n}$}, or that the sequence
$(x_1, . . ., x_n)$ is \textbf{$G_n$-invariant}, if $\varphi$ is invariant under the coaction $\alpha_n$, namely, for each $x\in A$,
$$(\varphi\otimes \text{id})\circ \alpha_n(x)=\varphi(x) \bm{1}_{C(G_n)}.$$
We refer the reader to \cite{ref5} for more information.

The Haar state $h_n$ with respect to $G_n$ is the unique state on $C(G_n)$ such that for any $a\in C(G_n)$, $$(h_n\otimes \operatorname{id})\circ \Delta(a)=(\operatorname{id}\otimes h_n)\circ\Delta (a)=h_n(a)\bm{1}_{C(G_n)}.$$ 
The law of the Haar state can be computed using the Weingarten formula. 
For $m\in \mathbb N$ , let $G_{m,n}^{\mathcal C}$ be the Gram matrix with entries indexed by the non-crossing partitions in $\mathcal C(m)$ defined by 
\begin{equation*}\label{eq12}
	G_{m,n}^{\mathcal C}(\pi,\sigma)=n^{\#({\pi\vee\sigma})} ,
\end{equation*}
where $\pi\vee\sigma$ is the smallest partition such that $\pi\leq \pi\vee\sigma$ and $\sigma\leq \pi\vee\sigma$ with respect to the
usual partial order in the lattice of partitions  $P(m)$. The Weingarten matrix $\operatorname{Wg}^{\mathcal C}_{m,n}$, if it exists, is defined as its inverse. The Haar state $h_n$ on $G_n$ is determined by the Weingarten formula:
\begin{equation}
	\label{Ewg}
	h_n(u_{i_1j_1}\cdots u_{i_mj_m})=
	\sum\limits_{\pi,\sigma\in{\mathcal C}(m)} \delta_{\pi}(\bm{i})\delta_{\sigma}(\bm{j})\text{Wg}^{\mathcal C}_{m,n}(\pi,\sigma).
\end{equation}
Banica and Collins proved that for $n\geq 2$, the Gram matrix with respect to $O_n^+$ is invertible. Furthermore, for $n\geq4$, the Gram matrix with respect to any categories of non-crossing partitions is invertible. We refer the reader to \cite{ref14, bk3} for more information.


\section{Finite de Finetti theorem for easy quantum groups}\label{sect3}


In this section, we present our finite de Finetti theorem for free easy quantum groups.
Our argument is motivated by the following observations:  
Let $A$ be a $\ast$-probability space, and let $ B$ be a unital $\ast$-subalgebra of $A$ with a $\varphi$-preserving conditional expectation $E$ onto $B$. The triple $(A,E,B)$ is called an \textbf{operator-valued non-commutative probability space}. We refer to \cite[Chapter 9]{bk1} for details.
The operator-valued free cumulant $\kappa_m^E:A^{m}\rightarrow {B}$ is determined by the operator-valued moment-cumulant formula :
\begin{equation}\label{Eopmc}
	E(b_0x_{i_1}b_1\cdots x_{i_m}b_m)=\sum_{\pi\in \operatorname{NC}(m)}\kappa^E_\pi(b_0x_{i_1}b_1,\cdots,x_{i_m}b_m).
\end{equation}

It is well known that the non-commutative random variables $(x_1,x_2,\dots,x_n)$ is a free semi-circular family with amalgamation over $B$  if and only if for all $m>0 $, $\pi\in\operatorname{NC}(m)$ and $b_i\in B$, $i=0,1,\ldots,m$,
$$\kappa^E_{\pi}[b_0x_{i_1}b_1,\ldots,x_{i_m}b_m]=0$$
except when $\pi\leq\ker{\bm{i}}$ and $\pi\in \operatorname{NC}_2(m)$  (the set of all non-crossing pairings of $[m]$). We denote $\operatorname{NC}_{2,\leq}(\bm{i})=\{\pi\in \operatorname{NC}_{2}(m): \pi\leq \ker(\bm{i})\}$.
Then by the operator-valued moment-cumulant formula (\ref{Eopmc}), we obtain 
\begin{equation}\label{eq10}
	\begin{aligned}
		E\left[x_{i_1}\cdots x_{j_m}\right]&=\sum_{\pi\in\operatorname{NC}_{2,\leq}(\bm{i})}\kappa^E_{\pi}(x_{i_1},\dots,x_{i_m})\\
		&=\sum_{\pi\in\operatorname{NC}_{2,\leq}(\bm{i})}\kappa^E_{\pi}(x_{1},\dots,x_{1}).
	\end{aligned}
\end{equation}
 Applying $\varphi$ to both sides of Eq. (\ref{eq10}), it follows that
\begin{equation*}
	\label{eq6}
	\varphi(x_{i_1}\cdots x_{i_m})=
	\sum\limits_{\pi\in\operatorname{NC}_{2,\leq}(\bm{i})}\varphi\left(\kappa^E_{\pi}(x_{1},\dots,x_{1})\right).
\end{equation*}
Then we have the following expression:
\begin{equation}
	\label{eq9}
	\varphi(x_{i_1}\cdots x_{i_m})=
	\sum\limits_{\pi\in\operatorname{NC}_{2,\leq}(\bm{i})}c^{(m)}_\pi,
\end{equation}
where $c^{(m)}_\pi=\varphi(\kappa^E_{\pi}(x_{1},\dots,x_{1}))$. Similar observations hold for operator-valued free independence with respect to other free easy quantum groups studied in \cite{ref5}. For $\bm{i}\in[n]^m$, we denote the set $\mathcal C_\leq(\bm{i}):=\{\pi\in \mathcal C(m): \pi\leq \ker(\bm{i})\}$.
The relations such as (\ref{eq9}) effect a much weaker independence than the freeness, and we will show that these exactly characterize the $G_n$-invariance. We may investigate further descriptions in terms of cumulants, inspired by the approach taken in \cite{ref4}.
More precisely, the following is our finite de Finetti theorem for free easy quantum groups.
\begin{thm}
	\label{th1}
	Given $n\geq 4$, let $(x_1, . . . , x_n)$ be a family of random variables in a $\ast$-probability space $(A, \varphi)$. Let $G_n$ be a free easy quantum group with the associated category of non-crossing partitions $\mathcal C$.
	Then the following are equivalent:
	\begin{enumerate}[\rm(i)]
		\item \label{th1-As1}The sequence $(x_1, . . ., x_n)$ is $G_n$-invariant. 
		\item \label{th1-As2}
		There exists a unital subalgebra $B$ of $A$ and a $\varphi$-preserving conditional expectation $E:A\rightarrow B$  such that, for any $m>0$,
		there is a family $(b_\pi)_{\pi\in \mathcal C(m)}\subset B$ satisfying, for all $\bm{i}\in [n]^m$,
		\begin{equation*}
			\label{eq5}
			E(x_{i_1}\ldots x_{i_m})=
			\sum\limits_{\pi\in\mathcal C_\leq(\bm{i})}b_\pi^{(m)} .
		\end{equation*}
		In particular, $E(x_{i_1}\cdots x_{i_m})$ is zero except if 
		$\mathcal C_{\leq}(\bm{i})\neq \varnothing$ and that $E(x_{i_1}\ldots x_{i_m})=E(x_{j_1}\cdots x_{j_m})$ as soon as $\ker(\bm{i})=\ker(\bm{j})$.
		\item \label{th1-As3} For any $m> 0$, there are scalars $(c_\pi^{(m)})_{\pi\in \mathcal C(m)}\subset\mathbb C$ such that, for all $\bm{i}\in [n]^m$,
		\begin{equation*}
			\label{eq1}
			\varphi(x_{i_1}\cdots x_{i_m})=
			\sum\limits_{\pi\in\mathcal C_\leq(\bm{i})}c_\pi^{(m)} .
		\end{equation*}
		In particular, the moments $\varphi(x_{i_1}\ldots x_{i_m})$ are zero except if 
		$\mathcal C_{\leq}(\bm{i})\neq \varnothing$ and 
		$$\varphi(x_{i_1}\cdots x_{i_m})=\varphi(x_{j_1}\cdots x_{j_m})$$ as soon as $\ker(\bm{i})=\ker(\bm{j})$.
		\item \label{th1-As4} For any $m>0$, there are scalars $(C_\pi^{(m)})_{\pi\in \mathcal C(m)}\subset\mathbb C$ such that, for all $\bm{i}\in [n]^m$,
		\begin{equation}\label{eq7}
			\kappa_{m}(x_{i_1},\ldots ,x_{i_m})=
			\sum\limits_{\pi\in\mathcal C_\leq(\bm{i})}C_\pi^{(m)}.
		\end{equation}
		In particular, the cumulants $\kappa_{m}(x_{i_1},\ldots ,x_{i_m})$ are zero except if 
		$\mathcal C_{\leq}(\bm{i})\neq \varnothing$ and $$\kappa_m(x_{i_1},\ldots ,x_{i_m})=\kappa_m(x_{j_1},\ldots, x_{j_m})$$ as soon as $\ker(\bm{i})=\ker(\bm{j})$.
	\end{enumerate}
\end{thm}
\begin{proof} 
	`(\ref*{th1-As1}) $\Rightarrow$ (\ref*{th1-As2})': 
Let $\mathcal P_{\bm{x}}^{\alpha_n}$ be the fixed point algebra of the coaction $\alpha_n$, namely, $$\mathcal P_{\bm{x}}^{\alpha_n}=\{x\in A:\alpha_n(x)=x\otimes \bm{1}_{C(G_n)}\}.$$  It is easy to check that
\begin{equation}\label{fix-point-alg}
	\mathcal P_{\bm{x}}^{\alpha_n}=\mathbb C\left\langle \sum_{\bm{i}\in [n]^{m}} \delta_{\pi}(\bm{i})x_{i_1}\cdots x_{i_{m}}: m\in\mathbb N, \pi\in\mathcal C(m)\right\rangle.
\end{equation}
Let $E: A\rightarrow\mathcal \mathcal P_{\bm{x}}^{\alpha_n}$ be a linear map defined by 
$E=(\operatorname{id}\otimes h)\circ\alpha_n.$
Then $E$ becomes a projection onto $\mathcal P_{\bm{x}}^{\alpha_n}$. If $(x_1,\ldots,x_n)$ is $G_n$-invariant, then $E$ is the $\varphi$-preserving conditional expectation onto $\mathcal P_{\bm{x}}^{\alpha_n}$. 
{For $\bm{i}=(i_1,\ldots,i_m)\in[n]^m$, applying the Weingarten formula (\ref{Ewg}), we have}
$$ \begin{aligned}
	E\left(x_{i_1}\cdots x_{i_{m}}\right)&=
	(\operatorname{id}\otimes h)\circ\alpha_n\left(x_{i_1}\cdots x_{i_{m}}\right)\\
	&=\sum_{\bm{j}\in[n]^m}x_{j_1}\cdots x_{j_{m}} h(u_{j_1i_1}\cdots u_{j_{m}i_{m}})\\
	&=\sum_{\bm{j}\in[n]^m}x_{j_1}\cdots x_{j_{m}}\sum_{\pi,\sigma\in\mathcal C(m)}\delta_{\pi}(\bm{i})\delta_{\sigma}(\bm{j})\operatorname{Wg}^{\mathcal C}_{m,n}(\pi,\sigma)\\ 
	&= \sum_{\pi\in\mathcal C(m)}\delta_{\pi}(\bm{i})b_{\pi}
\end{aligned}$$
where 
$$b_\pi=\sum_{\sigma\in\mathcal C(m)}\operatorname{Wg}^{\mathcal C}_{m,n}(\pi,\sigma) \sum_{\bm{j}\in[n]^m}\delta_{\sigma}(\bm{j})x_{j_1}\cdots x_{j_{m}},$$
which implies Assertion (\ref*{th1-As2}).

`(\ref*{th1-As2}) $\Rightarrow$ (\ref*{th1-As3})': For $\bm{i}\in[n]^m$, we have
$$\varphi(x_{i_1}\cdots x_{i_m})=\varphi\left(E(x_{i_1}\cdots x_{i_m})\right)
=
\sum\limits_{\pi\in\mathcal C_\leq(\bm{i})}\varphi\left(b_\pi\right).$$
	
	`(\ref*{th1-As3}) $\Rightarrow$ (\ref*{th1-As1})':
	For all $\pi \in \mathcal C(m)$, using Eq. (\ref{EqTP}), we have  
	$$\sum_{\bm{j}\in[n]^{m}}\delta_{\pi}(\bm{j})u_{j_1i_1}\cdots u_{j_{m}i_{m}}=\delta_{\pi}(\bm{i})\bm{1}_{C(G_n)}.$$
	Thus
	$$\begin{aligned}
		(\varphi\otimes \operatorname{id})\circ\alpha_n(x_{i_1}\cdots x_{i_m})&=\sum_{\bm{j}\in [n]^m}\varphi (x_{j_1}\cdots x_{j_m})u_{j_1i_1}\cdots u_{j_{m}i_{m}}\\
		&=\sum_{\bm{j}\in [n]^m}\sum_{\pi\in\mathcal C(m)} c_\pi^{(m)} \delta_\pi(\bm{j})u_{j_1i_1}\cdots u_{j_{m}i_{m}}\\
		&=\sum_{\pi\in\mathcal C(m)} c_\pi^{(m)} \delta_\pi(\bm{i}) \bm{1}_{C(G_n)}\\
		&=\varphi(x_{i_1}\cdots x_{i_m})\bm{1}_{C(G_n)}.
	\end{aligned}$$
	
	`(\ref*{th1-As4}) $\Rightarrow$ (\ref*{th1-As3})': Fix $\bm{i}\in[n]^m$. 
	For all $V=( k_1,\ldots,k_l)\subset [n]$, the value of
	  $\kappa_{l}(x_{i_{k_1}},\ldots,x_{i_{k_l}})$ only depends on $\ker(\bm{i}|_V)$. It is straightforward to verify that $\ker(\bm{i}|_V)=\ker(\bm{i})|_V$. 
	Therefore, by the moment-cumulant formula (\ref{Emc}) and Eq. (\ref{eq7}), for $\bm{j}\in [n]^m$, we have 
	$$\begin{aligned}
		\varphi(x_{j_1} \cdots x_{j_m})=\varphi(x_{i_1} \cdots x_{i_m})
		,
	\end{aligned}$$
	whenever $\ker(\bm{j})=\ker(\bm{i})$.
	Given $\pi\in \operatorname{NC}(m)$ and $\tau\in P(m)$, for each $\sigma\in\mathcal C(m)$ with $\sigma\leq \pi$, we have $\sigma|_{V}\in\mathcal C(|V|)$  for all $V\in \pi$, where $|V|$ is the number of elements in $V$. In addition, if $\sigma\leq\tau$, we have $\sigma|_{V}\leq \tau|_{V}$. Conversely, for all $V\in\pi$, if there exist  $\sigma|_V\in \mathcal C(|V|)$ with $\sigma|_V\leq \tau|_V$, then $\sigma=\bigcup_V \sigma|_V\leq \tau$. By the definition of the category of partitions (see e.g. \cite[Definition 1.3]{ref13}), if $\sigma_1\in \mathcal C(m_1)$ and $\sigma_2\in \mathcal C(m_2)$, then $\sigma_1\cup \sigma_2\in \mathcal C(m_1+m_2)$. It follows that $\sigma\in\mathcal C(m)$. Hence 
	$$ 	\prod\limits_{V\in\pi}
	\sum\limits_{\substack{\sigma\in \mathcal C(\abs{V})\\ \sigma\leq\tau|_{V}}}C_{\sigma}^{(m)}=\sum_{\substack{\sigma\leq\tau,\sigma\leq\pi\\ \sigma\in \mathcal C(m)}} \prod_{V\in\pi}C_{\sigma|_V}^{(m)}.$$
	For $\bm{i}\in [n]^m$, we denote $\tau=\ker(\bm{i})$. By Eq. (\ref{Emc}) and Eq. (\ref{Emc2}), we have
	$$\begin{aligned}
		\varphi\qty(x_{i_1}\cdots x_{i_m})
		&=\sum\limits_{\pi\in \operatorname{NC}(m)} \prod_{\substack{V\in\pi\\ V=(j_1,\ldots,j_l)}}\kappa_{l}(x_{i_{j_1}},\ldots,x_{i_{j_l}})\\
		&=\sum\limits_{\pi\in \operatorname{NC}(m)}
		\prod\limits_{V\in\pi}
		\sum\limits_{\substack{\sigma\in \mathcal C(\abs{V})\\ \sigma\leq\tau|_{V}}}C_{\sigma}^{(m)}\\
		&=\sum_{\pi\in\operatorname{NC}(m)}\sum_{\substack{\sigma\leq\tau,\sigma\leq\pi\\ \sigma\in \mathcal C(m)}} \prod_{V\in\pi}C_{\sigma|_V}^{(m)}\\
		&=\sum_{\substack{\sigma\in \mathcal C(m)\\ \sigma\leq\tau}}\sum_{\substack{\pi\in\operatorname{NC}(m)\\\pi\geq\sigma}} \prod_{V\in\pi}C_{\sigma|_V}^{(m)}.
	\end{aligned}$$
	For $\sigma\in\mathcal C(m)$, we denote
	$$c_\sigma^{(m)}=\sum_{\substack{\pi\in\operatorname{NC}(m)\\\pi\geq\sigma}} \prod_{V\in\pi}C_{\sigma|_V}^{(m)},$$
	which yields Assertion (ii).
	
	`(\ref*{th1-As3}) $\Rightarrow$ (\ref*{th1-As4})': Applying the  cumulant-moment formula \eqref{Ecm} to $\kappa_m$,  it follows that
	$$\kappa_m(x_{i_1}, \ldots, x_{i_{m}})=\sum\limits_{\pi\in \operatorname{NC}(m)}\mu_{P(m)}(\pi,1_m)\prod_{\substack{V\in\pi\\ V=(j_1,\ldots,j_l)}}\varphi(x_{i_{j_1}}\cdots x_{i_{j_l}}).$$
	As explained previously in `(\ref*{th1-As4}) $\Rightarrow$ (\ref*{th1-As3})', we have 
	$$ 	\prod\limits_{V\in\pi}
	\sum\limits_{\substack{\sigma\in \mathcal C(\abs{V})\\ \sigma\leq\tau|_{V}}}c_{\sigma}^{(m)}=\sum_{\substack{\sigma\leq\tau,\sigma\leq\pi\\ \sigma\in \mathcal C(m)}} \prod_{V\in\pi}c_{\sigma|_V}^{(m)}.$$
	It follows that
	$$\begin{aligned}
		\kappa_m(x_{i_1}, \ldots, x_{i_{m}})&=\sum\limits_{\pi\in \operatorname{NC}(m)}\mu_{P(m)}(\pi,1_m)\prod_{\substack{V\in\pi\\ V=(j_1,\ldots,j_l)}}\varphi(x_{i_{j_1}}\cdots x_{i_{j_l}})\\
		&=\sum_{\substack{\sigma\in \mathcal C(m)\\ \sigma\leq\tau}}\sum_{\substack{\pi\in\operatorname{NC}(m)\\\pi\geq\sigma}} \mu_{P(m)}(\pi,1_m) \prod_{V\in\pi}c_{\sigma|_V}^{(m)}.
	\end{aligned}$$
	Then we obtain
	$$C_\sigma^{(m)}=\sum_{\substack{\pi\in\operatorname{NC}(m)\\\pi\geq\sigma}}\mu_{P(m)}(\pi,\bm{1}_m) \prod_{V\in\pi}c_{\sigma|_V}^{(m)},$$
	which yields Assertion (iii).
\end{proof}

\begin{rk}\label{rk1}
		 (i) {From the arguments `(\ref*{th1-As3}) $\Leftrightarrow$ (\ref*{th1-As4})'}, if the assertions of Theorem \ref{th1} hold, then there are relations between the distinguished scalars $c_\sigma^{(m)}$ and $C_\sigma^{(m)}$ as follows:
		\begin{equation*}
			c_\sigma^{(m)}=\sum_{\substack{\pi\in\operatorname{NC}(m)\\\pi\geq\sigma}} \prod_{V\in\pi}C_{\sigma|_V}^{(m)}\text{ and } C_\sigma^{(m)}=\sum_{\substack{\pi\in\operatorname{NC}(m)\\\pi\geq\sigma}}\mu_{P(m)}(\pi,{1}_m) \prod_{V\in\pi}c_{\sigma|_V}^{(m)}.
		\end{equation*}
		
		(ii)  Let $(x_1,\ldots,x_n)$ be $G_n$-invariant. We denote $$D_n(m)=\{\pi\in P(m): \exists\bm{i}\in [n]^m \text{ such that } \ker{\bm{i}}=\pi\}.$$ 
		 Note that, by Theorem \ref{th1} (\ref*{th1-As3}), we have $\varphi(x_{i_1}\cdots x_{i_m})=\varphi(x_{j_1}\cdots x_{j_m})$ if $\ker(\bm{i})=\ker(\bm{j})$. So we may well define a map $\widetilde{\varphi}_m: D_n(m)\rightarrow \mathbb C$ given by
		$$\widetilde{\varphi}_m(\pi)=\varphi(x_{j_1}\ldots x_{j_m}),$$
		where  $(j_1,\ldots,j_m)\in [n]^m$ is any arbitrary  sequence such that $\ker(\bm{j})=\pi$. Assume that $m\leq n$, then we claim that the constants $c_{\sigma}^{(m)}$'s are uniquely determined, and for $\bm{i}\in [n]^m$, we have
		\begin{equation}\label{eq2}
			\varphi(x_{i_1}\cdots x_{i_m})=
			\sum\limits_{\pi\in \mathcal C_{\leq}(\bm{i})}\sum\limits_{\substack{\sigma\in \mathcal C(m)\\ \sigma\leq \pi}}{\widetilde\varphi_m(\sigma)\mu_{\mathcal C(m)}(\sigma,\pi)},
		\end{equation}
		where $\mu_{\mathcal C(m)}(\sigma,\pi)$ is the Möbius function on $\mathcal C(m)$. In fact,  
		if $m\leq n$, then $D_n(m)=P(m)$. Restricting $\widetilde\varphi_m$ to $\mathcal C(m)$, by Theorem \ref{th1} (\ref*{th1-As3}), we obtain that
		\begin{equation*}
			\label{eq4}
			\widetilde\varphi_m(\pi)=\sum\limits_{\substack{\sigma\in \mathcal C(m)\\ \sigma\leq \pi}} c_{\sigma}^{(m)},
		\end{equation*} 
		for all $\pi\in \mathcal C(m)$.
	From the equivalence between Eq. (\ref{eqi1}) and Eq. (\ref{eqi2}), we deduce that
		\begin{equation*}
			c_\pi^{(m)}=\sum\limits_{\substack{\sigma\in \mathcal C(m)\\ \sigma\leq \pi}}{\widetilde\varphi_m(\sigma)\mu_{\mathcal C(m)}(\sigma,\pi)},
		\end{equation*}
		which yields Eq. (\ref{eq2}). Similarly, let $\widetilde\kappa_m:D_n(m)\rightarrow \mathbb{R}$ be the map defined as 
		$$\widetilde{\kappa}_m(\pi)=\kappa_m(x_{j_1},\ldots, x_{j_m}),$$
		for an arbitrary  sequence $\bm{j}\in [n]^m$ such that $\ker(\bm{j})=\pi$.
		Then for all $m\leq n$ and $\bm{i}\in [n]^m$, we have
		\begin{equation}\label{eq3}
			\kappa_{m}(x_{i_1},\ldots ,x_{i_m})=
			\sum\limits_{\pi\in \mathcal C_{\leq}(\bm{i})}\sum\limits_{\substack{\sigma\in \mathcal C(m)\\ \sigma\leq \pi}}{\widetilde\kappa_m(\sigma)\mu_{\mathcal C(m)}(\sigma,\pi)}.
		\end{equation}
		In particular, let $(x_1,\ldots,x_n)$ be $O_n^+$-invariant. The category of non-crossing partitions associated with $O_n^+$ is $\operatorname{NC}_2$, where for any $\pi,\sigma\in \operatorname{NC}_2(m)$, the order relation $\pi\leq \sigma$ holds if and only if $\pi=\sigma$. Consequently,  {for all $m\leq 2n$ and $\bm{i}\in [n]^m$,} we have the inclusion $\operatorname{NC}_2(m)\subset D_n(m)$. By the previous analysis, it follows that,
		\begin{equation}\label{m-On-inv}
			\varphi(x_{i_1}\cdots x_{i_m})=
		\sum\limits_{\pi\in \operatorname{NC}_{2,\leq}(\bm{i})}{\widetilde\varphi_m(\pi)}
		\end{equation}
		and
		\begin{equation}\label{c-On-inv}
		\kappa_{m}(x_{i_1},\ldots ,x_{i_m})=
		\sum\limits_{\pi\in \operatorname{NC}_{2,\leq}(\bm{i})}{\widetilde\kappa_m(\pi)}.
		\end{equation}
\end{rk}

In the following, we discuss the asymptotic behaviors of the $G_n$-invariant random variables as $n\to\infty$. Note that a quantitative comparison with the operator-valued free distributions has been established in \cite[Theorem 4.8]{ref5}. Below, based on the techniques introduced in the main theorem, we will give alternative equivalent descriptions of (scalar-valued) asymptotic freeness for $S_n^+$ and $O_n^+$-invariant sequences. It is well known that a sequence  $(x_1^{(n)}, \ldots, x^{(n)}_n)$ is asymptotically free if and only if its mixed cumulants converge to zero as $n \rightarrow \infty$. The following simple corollaries illustrate that it suffices to focus on moments or cumulants associated with non-crossing partitions for some  $S_n^+$ and $O_n^+$-invariant sequences. Recall that the category $\mathcal C$ of partitions associated with $S_n^+$ (respectively, $O_n^+$) is $\operatorname{NC}$ (respectively, $\operatorname{NC}_2$).

\begin{cor}
	Let $n>0$ and let $(x_i^{(n)})_{i\in \mathbb N}$ be a sequence of self-adjoint elements in a $\ast$-probability space $(A_n,\varphi_n)$. For any $n\geq 4$, suppose $(x_1^{(n)},\ldots,x^{(n)}_n)$ is $S_n^+$-invariant. Then the following are equivalent:
	\begin{enumerate}[\rm(i)]
		\item\label{coro1-as2} For any $p\geq 2$, $x_1^{(n)},\ldots,x_p^{(n)}$ are asymptotically free as $n\rightarrow\infty$.
		\item\label{coro1-as3}  For any $m>0$ and $\bm{i}\in \mathbb N^{m}$, if  $\ker(\bm{i})\in\operatorname{NC}(m)$ and $\#\ker(\bm{i})>1$ (namely, the number of blocks in $\ker(\bm{i})$ is greater {than} $1$), then $\kappa_{m}(x_{i_1}^{(n)},\ldots, x_{i_{m}}^{(n)})\rightarrow 0$ as $n\rightarrow\infty$.
	\end{enumerate}
\end{cor}
\begin{proof}
	  `(\ref*{coro1-as2})$\Rightarrow$(\ref*{coro1-as3})':  Fix $m>0$. For $\pi\in P(m)$, we have $\#\pi=1$ if and only if $\pi=1_m$.
	  In other words, Assertion (\ref*{coro1-as2}) holds if and only if $\kappa_{m}(x_{j_1}^{(n)},\ldots ,x_{j_m}^{(n)})\rightarrow0$ whenever $\ker(\bm{j})\neq 1_m$. Let $p\geq m$. Then for each $\bm{i}\in \mathbb N^m$, there exists $\bm{j}\in [p]^m$, such that $\ker(\bm{j})=\ker(\bm{i})$. 
	Let $n$ be large enough. By Theorem \ref{th1} (\ref*{th1-As4}), we obtain $\kappa_{m}(x_{i_1}^{(n)},\ldots ,x_{i_m}^{(n)})=\kappa_{m}(x_{j_1}^{(n)},\ldots ,x_{j_m}^{(n)})$. 
	Therefore, if $\ker(\bm{i})\in \operatorname{NC}(m)$ and $\#\ker(\bm{i})>1$, we have
	$$\lim_{n\rightarrow\infty}\kappa_{m}\qty(x_{i_1}^{(n)},\ldots ,x_{i_m}^{(n)})=\lim_{n\rightarrow\infty}\kappa_{m}\qty(x_{j_1}^{(n)},\ldots ,x_{j_m}^{(n)})=0,$$
	which proves Assertion (\ref*{coro1-as3}).
	
	 `(\ref*{coro1-as3})$\Rightarrow$(\ref*{coro1-as2})':
	For any $m>0$,  when 
	$n$ is sufficiently large, since the sequence $(x_1^{(n)},\ldots,x^{(n)}_n)$ is $S_n^+$-invariant, it follows from Eq. \eqref{eq3} in Remark \ref{rk1} (ii) that
	\begin{equation}\label{eq13}
		\kappa_{m}^{(n)}\qty(x_{i_1}^{(n)},\ldots ,x_{i_m}^{(n)})=
		\sum\limits_{\pi\in \operatorname{NC}_{\leq}(\bm{i})}\sum\limits_{\substack{\sigma\in \operatorname{NC}(m)\\ \sigma\leq \pi}}{\widetilde\kappa_m^{(n)}(\sigma)\mu_{\operatorname{NC}(m)}(\sigma,\pi)},
	\end{equation}
	where $\widetilde\kappa_m^{(n)} : P(m) \rightarrow \mathbb{C}$ is defined by
	$$\widetilde\kappa_m^{(n)}(\ker(\bm{i}))=\kappa_m\qty(x^{(n)}_{i_1},\dots,x^{(n)}_{i_m}).$$
	Here $\widetilde\kappa_m^{(n)}$ is well-defined since for any $\sigma\in \operatorname{NC}(m)$, there exists $\bm{i}\in \N^m$ such that $\ker(\bm{i})=\sigma$. 
	Then by Assertion (\ref*{coro1-as3}) and Theorem \ref{th1} (\ref*{th1-As4}),
	we conclude that $\widetilde\kappa_m^{(n)}(\sigma)=\kappa_m(x^{(n)}_{i_1},\dots,x^{(n)}_{i_m})\rightarrow0$ for all $\sigma\in\operatorname{NC}(m)$ with $\sigma\neq 1_m$.
 	Note that $\#\sigma\geq\#\pi$ whenever $\sigma\leq \pi$. 
 	Therefore, from Eq. (\ref{eq13}) and the transitivity of the partial order `$\leq$', it follows that for any $p\geq 2$ and  $\bm{i}\in [p]^m$, 
 	$$\kappa_{m}^{(n)}\qty(x_{i_1}^{(n)},\ldots ,x_{i_m}^{(n)})\rightarrow 0,\text{  as } n\rightarrow \infty,$$  whenever $\ker(\bm{i})\neq 1_m$,
	which implies Assertion (i).
\end{proof}

\begin{cor}\label{coro1}
	Let $n>0$ and let $(x_i^{(n)})_{i\in \mathbb N}$ be a sequence of self-adjoint elements in a $\ast$-probability space $(A_n,\varphi_n)$ with variance $1$. For any $n\geq 2$, suppose $(x_1^{(n)},\ldots,x^{(n)}_n)$ is $O_n^+$-invariant. Then the following are equivalent:
	\begin{enumerate}[\rm(i)]
		\item For any $p\geq 2$, $$x_1^{(n)},\ldots,x_p^{(n)}\xrightarrow{distr.}s_1,\ldots,s_p, \text{ as }n\rightarrow\infty,$$ where $(s_1,\ldots,s_p)$ forms a semi-circular system in $(A,\varphi)$, namely, each $s_i$ ($1\leq i\leq p$) is a standard semicircular random variable and the variables are free with respect to $\varphi$ (see e.g. \cite[Definition 7.11]{bk4}).
		\item  For any $k\geq 2$ and $\bm{i}\in \N^{2k}$ such that $\ker(\bm{i})\in\operatorname{NC}_2(2k)$, $$\kappa_{2k}\left(x_{i_1}^{(n)},\ldots ,x_{i_{2k}}^{(n)}\right)\rightarrow 0, \text{ as } n\rightarrow\infty.$$
		\item For any  $k> 0$ and  $\bm{i}\in \N^{2k}$ such that $\ker(\bm{i})\in\operatorname{NC}_2(2k)$, $$\varphi_{n}\left(x_{i_1}^{(n)}\cdots x_{i_{2k}}^{(n)}\right)\rightarrow 1, \text{ as } n\rightarrow\infty.$$
		\item For all $k> 0$ and $\pi\in \operatorname{NC}_2(2k)$, 
		$$\varphi_n\qty(\frac{1}{n^k}\sum_{\bm{i}\in [n]^{2k}} \delta_{\pi}(\bm{i})x_{i_1}^{(n)}\cdots x_{i_{2k}}^{(n)})\rightarrow 1, \text{ as }n\rightarrow \infty.$$
	\end{enumerate}
\end{cor}
\begin{proof} 
	`(i) $\Rightarrow$ (ii)': 
	For any $k\geq 2$, we set $p\geq 2k$. If Assertion (i) holds, then for any $\bm{j}\in [p]^{2k}$, we have
	$$\lim_{n\rightarrow\infty}\kappa_{2k}\left(x_{j_1}^{(n)},\ldots ,x_{j_{2k}}^{(n)}\right)=\kappa_{2k}\left(s_{j_1},\ldots ,s_{j_{2k}}\right).$$
	Recall that the mixed cumulants of $s_1,\ldots,s_p$ are given by:
	\begin{equation}\label{sc-cumulant}
		\kappa_{2k}(s_{j_1},\ldots,s_{j_{2k}})=\left\{\begin{matrix}
			1, & \text{if } k=1\text{ and } j_1=j_2;\\
			0, & \text{otherwise}.
		\end{matrix}\right.
	\end{equation} 
	(see e.g. \cite[Examples 11.21]{bk4}).
	
	Given any $\bm{i}\in \N^{2k}$, there exists $\bm{j}\in [p]^{2k}$ such that $\ker(\bm{j})=\ker(\bm{i})$. By Theorem \ref{th1} (\ref*{th1-As4}) and Eq. \eqref{sc-cumulant}, 
	if  $\ker(\bm{i})\in \operatorname{NC}_2(2k)$, we have
	$$\lim_{n\rightarrow\infty}\kappa_{2k}\left(x_{i_1}^{(n)},\ldots ,x_{i_{2k}}^{(n)}\right)=\lim_{n\rightarrow\infty}\kappa_{2k}\left(x_{j_1}^{(n)},\ldots ,x_{j_{2k}}^{(n)}\right)=0,$$
	which implies Assertion (ii).
	
	`(ii) $\Rightarrow$ (iii)': 
	For any  $m\in \mathbb N$, let $n$ be sufficiently large. Since the sequence $(x_1^{(n)},\ldots,x^{(n)}_n)$ is $O_n^+$-invariant, it follows from Eq. (\ref{c-On-inv})  that the $m$-th order cumulant $\kappa_m(x_{i_1}^{(n)}\cdots x_{i_{m}}^{(n)})$ is given by  
	$$\kappa_m\qty(x_{i_1}^{(n)},\dots, x_{i_{m}}^{(n)})=\sum_{\pi\in \operatorname{NC}_{2,\leq}(\bm{i})} {\widetilde\kappa_m^{(n)}(\pi)}, $$
	where $\widetilde\kappa_m^{(n)}(\pi)=\kappa_m^{(n)}(x_{j_1}^{(n)},\dots, x_{j_{m}}^{(n)})$ for arbitrary sequence $\bm{j}\subset [n]^m$ such that $\ker(\bm{j})=\pi$.
	 In particular, $\kappa_m(x_{j_1}^{(n)}\cdots x_{j_{m}}^{(n)})=0$ whenever $m$ is odd. 
	 Together with Assertion (ii), it follows that, for all $m\neq 2$ and all  $\bm{j}\in \N^{m}$, 
	 $$\kappa_{m}\qty(x_{j_1}^{(n)}\cdots x_{j_{m}}^{(n)})\rightarrow 0,\text{ as }n\rightarrow \infty.$$ 
	  When $m=2$, It follows that
	$$\kappa_2\qty(x_{j_1}^{(n)},x_{j_2}^{(n)})=\begin{cases}
		\operatorname{Var}({x_1^{(n)}}), &\text{ if } j_1=j_2;\\
		0, &\text{ if } j_1\neq j_2.
	\end{cases}$$
	
	For any $k>0$ and $\bm{i}\in \N^{2k}$, there exists $\bm{j}\in \N^{2k}$ such that $\ker(\bm{j})= \ker(\bm{i})$.
	If $\ker(\bm{i})\in \operatorname{NC}_2(2k)$, then by Theorem \ref{th1} (\ref*{th1-As3}) and the moment-cumulant formula (\ref{Emc}), we obtain
	\begin{equation*}
		\begin{aligned} \lim_{n\rightarrow\infty}\varphi_n\left(x_{i_1}^{(n)}\cdots x_{i_{2k}}^{(n)}\right)
			&=\lim_{n\rightarrow\infty}\varphi_n\left(x_{j_1}^{(n)}\cdots x_{j_{2k}}^{(n)}\right)\\
			&=
			\lim_{n\rightarrow\infty}\sum_{\pi\in \operatorname{NC}(2k)}\kappa_\pi\left(x_{j_1}^{(n)},\ldots, x_{j_{2k}}^{(n)}\right)\\
			&=\lim_{n\rightarrow\infty}\sum_{\pi\in \operatorname{NC}_2(2k)}\kappa_\pi\left(x_{j_1}^{(n)},\ldots, x_{j_{2k}}^{(n)}\right)\\
			&=\lim_{n\rightarrow\infty}\kappa_{\ker(\bm{i})}\left(x_{j_1}^{(n)},\ldots, x_{j_{2k}}^{(n)}\right)\\
			&=\qty(\kappa_2\qty(s_{1},s_{1}))^k=1.
		\end{aligned}
	\end{equation*}	
	
	`(iii) $\Rightarrow$ (iv)': Given any  $m>0$, we assume $n$ is large enough. According to Remark \ref{rk1} (ii), we denote $\widetilde\varphi_m^{(n)}(\pi): P(m)\rightarrow \mathbb C$ defined as $\widetilde\varphi_m^{(n)}(\pi)=\varphi_n(x_{i_1}^{(n)}\cdots x_{i_{m}}^{(n)})$, where $\bm{i}$ is an arbitrary sequence in $[n]^m$ such that $\ker(\bm{i})=\pi$. 
	 For any $\pi\in P(m)$, there exists $\bm{j}\in \N^{m}$ such that $\ker(\bm{j})=\pi$. By Theorem \ref{th1} (\ref*{th1-As3}), we obtain 
	 $$\widetilde\varphi_m^{(n)}(\pi)=\varphi_n\qty(x_{j_1}^{(n)}\cdots x_{j_{m}}^{(n)}).$$
	  Moreover, if $m=2k$ ($k>0$) and
	$\pi\in\operatorname{NC}_2(2k)$, then by  Assertion (iii), we have 
	\begin{equation*}
		\widetilde\varphi_{2k}^{(n)}(\ker(\bm{i}))=\varphi_n\left(x_{j_1}^{(n)}\cdots x_{j_{2k}}^{(n)}\right)=1+o(1).
	\end{equation*}
	By Eq. (\ref{m-On-inv}), {it follows that}
	$$\begin{aligned}
		\varphi_n\qty(\frac{1}{n^k}\sum_{\bm{i}\in [n]^{2k}} \delta_{\pi}(\bm{i})x_{i_1}^{(n)}\cdots x_{i_{2k}}^{(n)})&=\frac{1}{n^k}\sum_{\bm{i}\in [n]^{2k}} \delta_{\pi}(\bm{i})\varphi_n\left(x_{i_1}^{(n)}\cdots x_{i_{2k}}^{(n)}\right)\\
		&=\frac{1}{n^k}\sum_{\bm{i}\in [n]^{2k}} \delta_{\pi}(\bm{i})\sum_{\sigma\in \operatorname{NC}_2(2k)}\delta_{\sigma}(\bm{i})\widetilde\varphi_{2k}^{(n)}(\sigma)\\
		&=\frac{1}{n^k}\sum_{\sigma\in \operatorname{NC}_2(2k)}\widetilde\varphi_{2k}^{(n)}(\sigma)\sum_{\bm{i}\in [n]^{2k}} \delta_{\pi}(\bm{i})\delta_{\sigma}(\bm{i})\\
		&=\frac{1}{n^k}\sum_{\sigma\in \operatorname{NC}_2(2k)}n^{\#(\pi\vee\sigma)}\left(1+o(1)\right).
	\end{aligned}$$
	For $\sigma,\pi\in \operatorname{NC}_2(2k)$, we have $\#(\pi\vee\sigma)\leq k$ and the equation holds if and only if $\pi=\sigma$. Therefore, we obtain
	$$\frac{n^{\#(\pi\vee\sigma)}}{n^k}\left(1+O(n^{-1})\right)\rightarrow\left\{
	\begin{matrix}
		1, & \pi=\sigma;\\
		0, & \text{otherwise},
	\end{matrix}
	\right.\quad\text{as } n\rightarrow \infty,$$ 
	which yields Assertion (iv).
	
	
	`(iv) $\Rightarrow$ (i)': 
		For each $m\in \mathbb N$, by Remark \ref{rk1} (ii), the moment $\varphi_n(x_{i_1}^{(n)}\cdots x_{i_{m}}^{(n)})$ is of the form given by Eq. (\ref{m-On-inv}), namely, 
	\begin{equation*}
		\varphi_n\qty(x_{i_1}^{(n)}\cdots x_{i_m}^{(n)})=
		\sum\limits_{\pi\in \operatorname{NC}_{2,\leq}(\bm{i})}{\widetilde\varphi_m^{(n)}(\pi)}.
	\end{equation*}
	In particular, $\varphi_n(x_{i_1}^{(n)}\cdots x_{i_{m}}^{(n)})=0$ whenever $m$ is odd. Assume $m=2k$ (k>0), then for all $p\geq 2$ and $\bm{i}\in [p]^{2k}$, we obtain
	$$ \begin{aligned}
		\varphi_n\left(x_{i_1}^{(n)}\cdots x_{i_{2k}}^{(n)}\right)
		&=(\varphi_n\otimes h)\circ\alpha_n\left(x_{i_1}^{(n)}\cdots x_{i_{2k}}^{(n)}\right)\\
		&=\sum_{\bm{j}\in[n]^{2k}}\varphi_n\left(x_{j_1}^{(n)}\cdots x_{j_{2k}}^{(n)}\right) h\qty(u_{j_1i_1}\cdots u_{j_{2k}i_{2k}})\\
		&=\sum_{\bm{j}\in[n]^{2k}}\varphi_n\left(x_{j_1}^{(n)}\cdots x_{j_{2k}}^{(n)}\right)\sum_{\pi,\sigma\in\operatorname{NC}_2(2k)} \delta_{\pi}(\bm{i}) \delta_{\sigma}(\bm{j})\operatorname{Wg}^{O_n^+}_{2k,n}(\pi,\sigma)\\
		&= \sum_{\pi,\sigma\in\operatorname{NC}_2(2k)} \delta_{\pi}(\bm{i}) \operatorname{Wg}^{O_n^+}_{2k,n}(\pi,\sigma)n^k \varphi_n\left(\frac{1}{n^k}\sum_{\bm{j}\in[n]^{2k}}\delta_{\sigma}(\bm{j})x_{j_1}^{(n)}\cdots x_{j_{2k}}^{(n)}\right).
	\end{aligned}$$
	By \cite[Proposition 7.2]{ref14} and Assertion (iv),  we obtain
	$$\operatorname{Wg}^{O_n^+}_{2k,n}(\pi,\sigma)n^k\rightarrow\left\{\begin{matrix}
		1 & \pi=\sigma\\
		0 & \text{otherwise}
	\end{matrix}\right. \quad \text{and} \quad  \varphi_n\left(\frac{1}{n^k}\sum_{\bm{j}\in[n]^{2k}}\delta_{\sigma}(\bm{j})x_{j_1}^{(n)}\cdots x_{j_{2k}}^{(n)}\right)\rightarrow 1.$$
	It follows that, for each $m\in \mathbb N$ and $\bm{i}\in [p]^m$, we have
	$$\lim_{n\rightarrow\infty}\varphi_n\left(x_{i_1}^{(n)}\cdots x_{i_{m}}^{(n)}\right)=\sum_{\pi\in\operatorname{NC}_{2,\leq}(\bm{i})}1.$$
	In addition, applying the moment-cumulant formula (\ref{Emc}) together with the values of mixed cumulants  for  free semicicular random variables (\ref{sc-cumulant}), we obtain
	$$\varphi\left(s_{i_1}\cdots s_{i_{m}}\right)=\sum_{\pi\in\operatorname{NC}(m)}\kappa_{m}(s_{i_1},\ldots,s_{i_m})=\sum_{\pi\in\operatorname{NC}_{2,\leq}(\bm{i})}1.$$
	Thus, Assertion (i) follows.
\end{proof}


\section{Infinite de Finetti theorem for  easy quantum groups}\label{sect4}

In this section, we refine the previous arguments to obtain an infinite de Finetti theorem for free easy quantum groups. 
Let $(A,\varphi)$ be a $\ast$-probability space generated by a sequence of self-adjoint elements $(x_i)_{i\in\mathbb N}$. 
For a category of non-crossing partitions $\mathcal C$, we denote by $G_n(\mathcal C)$ (or simply $G_n$ if no confusion arises) the associated free easy quantum group acting on $n$ variables. 
We call $(x_i)_{i\in\mathbb N}$ \textbf{${G}$-invariant} if $(x_1,\ldots,x_n)$ is $G_n$-invariant for all $n\in\mathbb N$. 

In the infinite setting,  let $(x_i)_{i\in\mathbb N}$ be a $G$-invariant sequence in the $\ast$-probability space $(A,\varphi)$. For each $\pi\in P(m)$, there exists $\bm{i}=(i_1,\ldots,i_m)\in \nat^m$, such that $\ker{(\bm{i})}=\pi$. 
Extending the notation from the  Remark \ref{rk1}  (ii), we define the map $\widetilde{\varphi}_m: P(m)\rightarrow\mathbb C$ by $\widetilde{\varphi}_m(\pi)=\varphi(x_{i_1}\cdots x_{i_m})$, where $(i_1,\ldots,i_m)$ is any arbitrary  sequence in $\mathbb N^m$ satisfying $\ker(\bm{i})=\pi$. $\widetilde{\varphi}_m$ is well defined  follows from Theorem \ref{th1} which guarantees that $\varphi(x_{i_1}\cdots x_{i_m})=\varphi(x_{j_1}\cdots x_{j_m})$ whenever $\ker(\bm{i})=\ker(\bm{j})$.
 An analogous result holds for the cumulant function $\widetilde{\kappa}_m$. With these definitions in place, we now state the following theorem.
\begin{thm}
	\label{th2}
	Let $(x_i)_{i\in \mathbb N}$ be a family of self-adjoint random variables in the $\ast$-probability space $(A, \varphi)$.  Let $G_n$ and $\mathcal C$ be as above. Then
	the following are equivalent:
	\begin{itemize}
		\item	
		[(i)] The sequence $(x_i)_{i\in \mathbb N}$ is $G$-invariant. 
		\item
		[(ii)] For any $m>0$ and $\bm{i}\in \nat^m$, 
		\begin{equation}\label{eqp}
			\varphi(x_{i_1}\cdots x_{i_m})=
			\sum\limits_{\pi\in \mathcal C_{\leq}(\bm{i})}\sum\limits_{\substack{\sigma\in \mathcal C(m)\\ \sigma\leq \pi}}{\widetilde\varphi_m(\sigma)\mu_{\mathcal C(m)}(\sigma,\pi)},
		\end{equation}
		In particular, $\varphi(x_{i_1}\cdots x_{i_m})$ is zero except if   
		$\mathcal C_{\leq}(\bm{i})\neq \varnothing$.
		\item[(iii)] For any $m>0$ and $\bm{i}\in \nat^m$, 
		\begin{equation}\label{eqc1}
			\kappa_{m}(x_{i_1},\ldots ,x_{i_m})=
			\sum\limits_{\pi\in \mathcal C_{\leq}(\bm{i})}\sum\limits_{\substack{\sigma\in \mathcal C(m)\\ \sigma\leq \pi}}{\widetilde\kappa_m(\sigma)\mu_{\mathcal C(m)}(\sigma,\pi)}.
		\end{equation}
		In particular, $\kappa_{m}(x_{i_1},\ldots ,x_{i_m})$ is zero except if   
		$\mathcal C_{\leq}(\bm{i})\neq \varnothing$.
	\end{itemize}
\end{thm}
\begin{proof}

	If $(x_i)_{i\in \mathbb N}$ is $G$-invariant, then for each $n>0$, the tuple $(x_1,\ldots,x_n)$ is $G_n$-invariant. Let $A_n$ denote the subalgebra of $A$ generated by $(x_1,\ldots,x_n)$ and define $\varphi^{(n)}=\varphi|_{A_n}$. For $m,n\in \mathbb N$,
we retain the notation $D_n(m)$ from Remark \ref{rk1} and similarly define $\wt\varphi^{(n)}_m :D_n(m)\rightarrow \mathbb C$ as the functional determined by $\varphi^{(n)}$.
 If $n\geq 2m$, then $D_n(m)=P(m)$.
	By Theorem \ref{th1}, there exist scalars $c_\pi^{(n,m)}$ and $C_\pi^{(n,m)}$ ($\pi\in \mathcal C(m)$) such that, for any $\sigma\in \mathcal C(m)$,
	\begin{equation}\label{eq11}
		\widetilde\varphi^{(n)}_m(\sigma)=
		\sum\limits_{\substack{\pi\in\mathcal C(m)\\ \pi\leq \sigma}}c_\pi^{(n,m)}
		\quad \text{and}\quad \widetilde\kappa_{m}^{(n)}(\sigma)=
		\sum\limits_{\substack{\pi\in\mathcal C(m)\\ \pi\leq \sigma}}C_\pi^{(n,m)}.
	\end{equation}
	Thus, if $n\geq2m$, by the equivalence between Eq. (\ref{eqi1}) and Eq. (\ref{eqi2}), we have \begin{equation}\label{eq14}
		c_\pi^{(n,m)}=\sum\limits_{\substack{\sigma\in \mathcal C(m)\\ \sigma\leq \pi}}{\widetilde\varphi_m^{(n)}(\sigma)\mu_{\mathcal C(m)}(\sigma,\pi)}
		\quad \text{and} \quad
		C_\pi^{(n,m)}=\sum\limits_{\substack{\sigma\in \mathcal C(m)\\ \sigma\leq \pi}}{\widetilde\kappa_m^{(n)}(\sigma)\mu_{\mathcal C(m)}(\sigma,\pi)}.
	\end{equation} 
	The definition of the functionals $\widetilde\varphi^{(n)}_m$ and $\widetilde\varphi_m$ immediately implies that, for any $\pi\in D_n(m)$,
	 $\widetilde\varphi^{(n)}_m(\pi)=\wt\varphi_m(\pi)$.	
	Together with Eq. (\ref{eq14}), we obtain Eq. (\ref{eqp}) and Eq. (\ref{eqc1}), 
	{which yield Assertions (ii) and (iii).}
	
	Conversely, if the Assertion (ii) (respectively, Assertion (iii)) holds, then for each $m,n>0$, the moments (respectively, the cumulants) of $(x_1,\ldots,x_n)$ is of the form Eq. (\ref{eq11}),
	where 
	$$ c_\pi^{(m)}=\sum\limits_{\substack{\sigma\in \mathcal C(m)\\ \sigma\leq \pi}}{\widetilde\varphi_m(\sigma)\mu_{\mathcal C(m)}(\sigma,\pi)}\quad \text{and}\quad C_\pi^{(m)}=\sum\limits_{\substack{\sigma\in \mathcal C(m)\\ \sigma\leq \pi}}{\widetilde\kappa_m(\sigma)\mu_{\mathcal C(m)}(\sigma,\pi)}.$$
	Thus by Theorem \ref{th1}, the sequence $(x_1, \ldots, x_n)$ is $G_n$-invariant.
\end{proof}


For $\times\in\{o,b,s,h\}$, we denote by $G_n^{\times}$ the free easy quantum groups $O_n^+$, $B_n^+$, $S_n^+$ and $H_n^+$, respectively, and by $\mathcal C^\times$ the category of non-crossing partitions associated with $G_n^\times$. Let $(\mathcal M,\varphi)$ be a $W^*$-probability space, namely, $\mathcal M$ is a von Neumann algebra and $\varphi$ is a faithful normal state ($\varphi$ is not necessarily a trace.) Let $(x_i)_{i \in \mathbb{N}}$ be a sequence of self-adjoint elements in $\mathcal M$.
It is shown in \cite[Theorem 1]{ref5} that, for $\times\in\{o,b,s,h\}$, there exists a von Neumann subalgebra $\mathcal B$ of $\mathcal M$, such that the  $G^{\times}$-invariance of $(x_i)_{i \in \mathbb{N}}$ can be characterized by its operator-valued freeness:
\begin{itemize}
	\item [(a)] If $\times=o$, $(x_i)_{i \in \mathbb{N}}$ is a  $\mathcal B$-valued free semi-circular family with  mean zero and common variance.
	\item [(b)] If $\times=h$, $x_i$, ${i \in \mathbb{N}}$, have even and identical distributions and are free with amalgamation over $\mathcal B$.
	\item [(c)] If $\times=b$, $(x_i)_{i \in \mathbb{N}}$ is a  $\mathcal B$-valued free semi-circular family with common mean and variance.
	\item [(d)] If $\times=s$, $x_i$, ${i \in \mathbb{N}}$, have identical distributions and are free with amalgamation over $\mathcal B$.
\end{itemize}
For convenience,  we will say that $(x_i)_{i \in \mathbb{N}}$ has the $\times$-distribution with respect to $\mathcal B$.
Combining Theorem \ref{th2} with \cite[Theorem 1]{ref5}, we can derive the following equivalent conditions for operator-valued freeness, bypassing the quantum group invariance. 

\begin{cor}\label{coro3}
	Consider the $W^*$-probability space $(\mathcal M, \varphi)$, which is endowed with a faithful state, generated by an infinite sequence of self-adjoint random variables $(x_i)_{i\in \mathbb N}$. Then, for $\times\in\{o,b,s,h\}$, the following are equivalent:
	\begin{itemize}
		\item [(i)] For any $m>0$ and $\bm{i}\in \nat^m$, 
		\begin{equation*}\label{eqc}
			\kappa_{m}(x_{i_1},\ldots ,x_{i_m})=
			\sum\limits_{\pi\in \mathcal C_{\leq}^\times(\bm{i})}\sum\limits_{\substack{\sigma\in \mathcal C(m)\\ \sigma\leq \pi}}{\widetilde\kappa_m(\sigma)\mu_{\mathcal C^\times(m)}(\sigma,\pi)}.
		\end{equation*}
		In particular, $\kappa_{m}(x_{i_1},\ldots ,x_{i_m})$ is zero except if   
		$\mathcal C_{\leq}^\times(\bm{i})\neq \varnothing$.
		\item[(ii)] 
		There is a von Nuemann subalgebra $1\subset \mathcal B\subset \mathcal M$ and a $\varphi$-preserving conditional expectation $E:\mathcal M\rightarrow\mathcal  B$ such that 
		$(x_i)_{i \in \mathbb{N}}$ has the $\times$-distribution with respect to $\mathcal B$.
		\item [(iii)] 
		There is a von Nuemann subalgebra $1\subset \mathcal B\subset \mathcal M$ and a $\varphi$-preserving conditional expectation $E:\mathcal M\rightarrow\mathcal  B$ 
		such that, for any $m>0$,
		there is a family $(b_\pi)_{\pi\in \mathcal C(m)}\subset \mathcal B$ satisfying, for all $\bm{i}\in \nat^m$,
		\begin{equation}\label{eq8}
			E\left[x_{i_1}\cdots x_{i_m}\right]=
			\sum\limits_{\pi\in\mathcal C^\times_{\leq}(\bm{i})}b_\pi. 
		\end{equation}
		In particular, $E(x_{i_1}\cdots x_{i_m})$ are zero except if 
		$\mathcal C_{\leq}^\times(\bm{i})\neq \varnothing$ and $E(x_{i_1}\cdots x_{i_m})=E(x_{j_1}\cdots x_{j_m})$ as soon as $\ker(\bm{i})=\ker(\bm{j})$.
	\end{itemize}
\end{cor}
\begin{proof}
	 `(i)$\Leftrightarrow$(ii)': By Theorem \ref{th2} and \cite[Theorem 1]{ref5}, we have Assertions (i) and  (ii) are both equivalent to the $G^\times$-invariance of $(x_i)_{i\in \mathbb N}$.
	 
	`(ii)$\Rightarrow$(iii)': 
	It is well known that Assertion (ii) holds if and only if, 
	for all $m>0 $, $\pi\in\C^\times(m)$ and $b_i\in B$, $i=0,1,\ldots,m$,
	$$\kappa^E_{\pi}[b_0x_{i_1}b_1,\ldots,x_{i_m}b_m]=0$$
	except when $\pi\in \C_\leq^\times(\bm{i})$. 
	Then it follows from the operator-valued moment-cumulant formula \eqref{Eopmc} that
 $$\begin{aligned}
		E\left[x_{i_1}\cdots x_{i_m}\right]&=
	\sum\limits_{\pi\in\mathcal C^\times_{\leq}(\bm{i})}\kappa^{E}_\pi(x_{i_1},\cdots,x_{i_m})\\
	&=\sum\limits_{\pi\in\mathcal C^\times_{\leq}(\bm{i})}\kappa^{E}_\pi(x_{1},\cdots,x_{1}),
	\end{aligned}$$ 
	where the second equality holds because the $\mathcal B$-valued distribution of $x_i$ ($i\in \N$) are identical.
	Since each $\kappa^{E}_\pi(x_{1},\cdots,x_1)$ lies in $\mathcal B$,
	we may define 
	$b_\pi=\kappa^{E}_\pi(x_{1},\cdots,x_1)$,
	 which  implies Assertion (iii).
		
	`(iii)$\Rightarrow$(i)': Applying $\varphi$ to both sides of (\ref{eq8}), 
	we have 
	$$
	\varphi(x_{i_1}\ldots x_{i_m})=
	\sum\limits_{\pi\in\mathcal C^\times_\leq(\bm{i})}\varphi(b_\pi).
	$$
	By Remark \ref{rk1}, we have that $\varphi(x_{i_1}\ldots x_{i_m})$ is of the form (\ref{eqp}),
	namely, 
	\begin{equation*}
		\varphi(x_{i_1}\cdots x_{i_m})=
		\sum\limits_{\pi\in \mathcal C_{\leq}^\times(\bm{i})}\sum\limits_{\substack{\sigma\in \mathcal C^\times(m)\\ \sigma\leq \pi}}{\widetilde\varphi_m(\sigma)\mu_{\mathcal C^\times(m)}(\sigma,\pi)}.
	\end{equation*}
	Thus Assertion (i) follows immediately from Theorem \ref{th2}.
\end{proof}

\begin{rk}\label{rk3}
	Let $n$ be an integer and let $A_u(n)$ denote the universal $C^*$-algebra generated by $n^2$ elements $(v_{ij})_{1\leq i,j\leq n}$ such that 
	both $v$ and $v^t$ are unitary. The pair $(A_u(n),v)$ is known as the  quantum unitary group and is commonly denoted by $U_n^+$. We refer to \cite{ref8} for details. The quantum unitary group $U_n^+$ can be associated with the category of non-crossing pair coloured partitions, denoted by $\operatorname{NC}_2^{\circ\bullet}$. We refer the reader to  \cite{bk3} to further more about coloured partition, to \cite{ref15} for the coaction of $U_n^+$ on $\ast$-algebra and to \cite{ref14} for the Haar state with respect to $U_n^+$.
	Theorems \ref{th1} and Theorem \ref{th2} can be extended verbatim to a unitary version by simply replacing $G_n$ with $U_n^+$ and $\operatorname{NC}_\times(m)$ with $\operatorname{NC}_2^{\circ\bullet}$.
\end{rk} 

\bigskip \n{\bf Acknowledgements.}  The author is indebted to his advisor Simeng Wang for invaluable discussions, meticulous feedback, and patient guidance throughout this work. Special thanks also go to Sheng Yin  for his insightful suggestions and constant inspiration. This work is partially supported by the NSF of China (No. 12301161, No. 12031004, No. W2441002)


 \bibliographystyle{plain}
\bibliography{reference}


\end{document}